\documentclass[12pt]{amsart}
\usepackage{amsmath,amssymb,amsbsy,amsfonts,amsthm,latexsym,mathabx,
            amsopn,amstext,amsxtra,euscript,amscd,stmaryrd,mathrsfs,
            cite,array,mathtools,enumerate}

\usepackage{url}
\usepackage[colorlinks,linkcolor=blue,anchorcolor=blue,citecolor=blue,backref=page]{hyperref}

\usepackage{color}

\usepackage[norefs,nocites]{refcheck}

\usepackage{color}
\usepackage{float}

\hypersetup{breaklinks=true}

\usepackage[english]{babel}

\definecolor{olive}{rgb}{0.3, 0.4, .1}
\definecolor{dgreen}{rgb}{0.,0.5,0.}

\usepackage{enumitem}

\usepackage{algorithm}
\usepackage[noend]{algpseudocode}
\usepackage{varwidth}

\begin{document}

\newtheorem{theorem}{Theorem}
\newtheorem{lemma}[theorem]{Lemma}
\newtheorem{claim}[theorem]{Claim}
\newtheorem{cor}[theorem]{Corollary}
\newtheorem{prop}[theorem]{Proposition}
\newtheorem{definition}{Definition}
\newtheorem{question}[theorem]{Open Question}
\newtheorem{example}[theorem]{Example}

\numberwithin{equation}{section}
\numberwithin{theorem}{section}
\numberwithin{algorithm}{section}

 \newcommand{\F}{\mathbb{F}}
\newcommand{\K}{\mathbb{K}}
\newcommand{\D}[1]{D\(#1\)}
\def\scr{\scriptstyle}
\def\\{\cr}
\def\({\left(}
\def\){\right)}
\def\[{\left[}
\def\]{\right]}
\def\<{\langle}
\def\>{\rangle}
\def\fl#1{\left\lfloor#1\right\rfloor}
\def\rf#1{\left\lceil#1\right\rceil}
\def\le{\leqslant}
\def\ge{\geqslant}
\def\eps{\varepsilon}
\def\mand{\qquad\mbox{and}\qquad}

\newcommand{\commA}[1]{\marginpar{%
\begin{color}{red}
\vskip-\baselineskip %raise the marginpar a bit
\raggedright\footnotesize
\itshape\hrule \smallskip A: #1\par\smallskip\hrule\end{color}}}

\newcommand{\commD}[1]{\marginpar{%
\begin{color}{blue}
\vskip-\baselineskip %raise the marginpar a bit
\raggedright\footnotesize
\itshape\hrule \smallskip D: #1\par\smallskip\hrule\end{color}}}

\newcommand{\commI}[1]{\marginpar{%
\begin{color}{magenta}
\vskip-\baselineskip %raise the marginpar a bit
\raggedright\footnotesize
\itshape\hrule \smallskip I: #1\par\smallskip\hrule\end{color}}}

\newcommand{\commL}[1]{\marginpar{%
\begin{color}{olive}
\vskip-\baselineskip %raise the marginpar a bit
\raggedright\footnotesize
\itshape\hrule \smallskip L: #1\par\smallskip\hrule\end{color}}}

\newcommand{\Fq}{\mathbb{F}_q}
\newcommand{\Fp}{\mathbb{F}_p}
\newcommand{\Disc}[1]{\mathrm{Disc}\(#1\)}
\newcommand{\Res}[1]{\mathrm{Res}\(#1\)}
%\newcommand{\Nm}[1]{\mathrm{Norm}_{\F{q^k/\Fq}}(#1)}

%%%%%%%%%%%%%%%%%%%%%%%%%
% Alphabet calligraphie %
%%%%%%%%%%%%%%%%%%%%%%%%%
\def\cA{{\mathcal A}}
\def\cB{{\mathcal B}}
\def\cC{{\mathcal C}}
\def\cD{{\mathcal D}}
\def\cE{{\mathcal E}}
\def\cF{{\mathcal F}}
\def\cG{{\mathcal G}}
\def\cH{{\mathcal H}}
\def\cI{{\mathcal I}}
\def\cJ{{\mathcal J}}
\def\cK{{\mathcal K}}
\def\cL{{\mathcal L}}
\def\cM{{\mathcal M}}
\def\cN{{\mathcal N}}
\def\cO{{\mathcal O}}
\def\cP{{\mathcal P}}
\def\cQ{{\mathcal Q}}
\def\cR{{\mathcal R}}
\def\cS{{\mathcal S}}
\def\cT{{\mathcal T}}
\def\cU{{\mathcal U}}
\def\cV{{\mathcal V}}
\def\cW{{\mathcal W}}
\def\cX{{\mathcal X}}
\def\cY{{\mathcal Y}}
\def\cZ{{\mathcal Z}}

\newcommand{\sR}{\ensuremath{\mathscr{R}}}
\newcommand{\sDI}{\ensuremath{\mathscr{DI}}}
\newcommand{\DI}{\ensuremath{\mathrm{DI}}}
\newcommand{\sD}{\ensuremath{\mathscr{D}}}

\newcommand{\Nm}[1]{\mathrm{Norm}_{\,\F_{q^k}/\Fq}(#1)}

\def\Tr{\mbox{Tr}}
\newcommand{\rad}[1]{\mathrm{rad}(#1)}

\newcommand{\Orb}[1]{\mathrm{Orb}\(#1\)}
\newcommand{\aOrb}[1]{\overline{\mathrm{Orb}}\(#1\)}

\title[Counting of Dynamically Irreducible Polynomials]
{On the Complexity of Exact Counting of Dynamically Irreducible Polynomials}

\author[D. G\'omez-P\'erez]{Domingo G\'omez-P\'erez}
\address{D.G.-P.: Department of Mathematics, University of Cantabria, Santander 39005, Spain}
\email{domingo.gomez@unican.es}

\author[L. M{\'e}rai]{L{\'a}szl{\'o} M{\'e}rai}
\address{L.M.: Johann Radon Institute for
Computational and Applied Mathematics,
Austrian Academy of Sciences, Altenberger Stra\ss e 69, A-4040 Linz, Austria}
\email{laszlo.merai@oeaw.ac.at}
%
% \author[A. Ostafe] {Alina Ostafe}
% \address{A.O.: School of Mathematics and Statistics, University of New South Wales.
% Sydney, NSW 2052, Australia}
% \email{alina.ostafe@unsw.edu.au}

\author[I.~E.~Shparlinski]{Igor E. Shparlinski}
\address{I.E.S.: School of Mathematics and Statistics, University of New South Wales.
Sydney, NSW 2052, Australia}
\email{igor.shparlinski@unsw.edu.au}

\pagenumbering{arabic}

\maketitle

\begin{abstract} We give an efficient algorithm to enumerate all sets of $r\ge 1$ quadratic polynomials
over a finite field,  which remain irreducible under iterations and compositions.
\end{abstract}

\section{Introduction}

For a finite field $\Fq$ and a polynomial $f\in \Fq[X]$ we define
the sequence:
$$
  f^{(0)}(X)  = X, \qquad f^{(n)}(X)  = f\(f^{(n-1)}(X)\), \quad
  n =1, 2, \ldots\,.
$$
The polynomial $f^{(n)}$ is called the $n$-th iterate of the polynomial $f$.

Following the terminology established in~\cite{Ali,AyMcQ,Jon,JB}, we say that a polynomial $f\in \Fq[X]$
is {\it stable\/} if all iterates $ f^{(n)}(X)$, $n =1,2 , \ldots$,  are irreducible over $\Fq$.
However here, we prefer to use a more informative terminology of Heath-Brown and Micheli~\cite{H-BM} and instead we  call such polynomials {\it dynamically irreducible\/}.

Let $q$ be and odd prime power, and as in~\cite{JB}, for a quadratic polynomial $f(X)=aX^2+bX+c\in\Fq[X]$, $a\neq 0$ we define $\gamma=-b/(2a)$ as the unique critical point of $f$ (that is, the zero of the derivative $f'$). We remark that for $q$ even, it is known that there
does not exist quadratic stable polynomials~\cite{Ahmadi}.

Let $\sDI_q$ be the  set of dynamically irreducible quadratic polynomials over
a finite field of $q$ elements $\Fq$ and let $\DI_q = \# \sDI_q$ be their number.

Ostafe and Shparlinski~\cite{OstShp} have  shown that  for  a quadratic
polynomial $f \in \Fq[X]$ one can test whether
$f\in \sDI_q$ in time $q^{3/4+o(1)}$, see Lemma~\ref{lem:CritOrb-M} below.

G\'omez-P\'erez and  Nicol{\'a}s~\cite{GoNi}, developing some
ideas from~\cite{OstShp}, have proved that for an odd prime power $q$
we have
\begin{equation}\label{eq:stable-bound}
\frac{(q-1)^2}{4} \le \DI_q  = O(q^{5/2}\log q),
\end{equation}
where the implied constant is absolute.

These results have been generalized in~\cite{GNOS}, which in particular gives 
an upper bound %is given 
on the number of dynamically irreducible polynomials of degree $d\ge 2$ over $\Fq$.
 
Here we consider the question of constructing the set $\sDI_q$ and exactly evaluating
its cardinality $\DI_q$.
Trivially, using the above test from~\cite{OstShp}, one can construct the set $\sDI_q$ in time $q^{15/4+o(1)}$.
It is possible to calculate $\DI_q$ faster, in time $q^{11/4+o(1)}$ if one uses the
correspondence between arbitrary and monic dynamically irreducible polynomials, see Lemma~\ref{lem:monic}
 below.  Here we give a more efficient algorithm.

\begin{theorem}\label{thm:counting}
Let $q$ be an odd prime power. Then there exists an algorithm which computes
$\DI_q$ in time $q^{9/4+o(1)}$ and  constructs the set $\sDI_q$ in time $q^{5/2+o(1)}$.
\end{theorem}
We give the pseudocode of the algorithm in Theorem~\ref{thm:counting} in Algorithm~\ref{alg:1}.

We also study an analogous problem in semigroups generated by several polynomials under the composition.

Let $f_1(X),\ldots, f_r(X)\in\Fq[X]$ be polynomials of positive degree. The set $\{f_1(X),\ldots,f_r(X)\}$ is called \emph{dynamically irreducible} if all the iterates $f_{i_1}\circ \ldots \circ f_{i_n}$, for $i_1,\ldots,i_n\in\{1,\ldots, r\}$ and $n \geq 1$ are irreducible.

Ferraguti, Micheli and Schnyder~\cite{FMS1} have characterized the sets of monic quadratic polynomial to be  dynamically
irreducible in terms of the unique critical points of the polynomials. We also note that the  subsequent work~\cite{FMS2} 
gives a representation  of 
the set dynamically
irreducible polynomials via  finite automata. 
Furthermore, Heath-Brown and Micheli~\cite{H-BM} have given an algorithm to test whether a set of monic polynomials is
dynamically irreducible.

Here we consider the question of how to construct the set $\sDI_q(r)$ of all sets of $r$ arbitrary
pairwise distinct quadratic polynomials over $\Fq$ which are dynamically irreducible and find its cardinality 
$$
\DI_q(r)=\#\sDI_q(r).
$$
In particular, $\DI_q(1) = \DI_q$.

Furthermore, we use $\sDI^*_q(r)$ to denote the subset  $\sDI_q(r)$
consisting of \emph{monic} quadratic polynomials and also use $\DI_q^*(r) =  \#\sDI_q^*(r)$
for its cardinality.

Let $M(q)$ and $M^*(q)$ be the size of the largest set of dynamically irreducible
non-monic and monic quadratic polynomials, respectively.
Because we consider both  the monic and non-monic cases, $M^*(q)$ in this paper correspond to  $M(q)$  in the paper by
 Heath-Brown and Micheli~\cite{H-BM}, who have proven that $M^*(q) \le 32q(\log q)^4$ in general, 
while for infinitely many finite fields $M^*(q) \ge 0.5 (\log q)^2$.

It is easy that the bound~\eqref{eq:Gamma3} below implies
$M(q) \le q^{3/2+o(1)}$. On the other hand, in Example~\ref{ex:DI poly}
we present  an explicit family of  quadratic  polynomials which shows that  $M(q)\ge(q-1)/2$
for infinitely many $q$ (namely for those for which $-1$ a square in $\F_q$).
Thus in the case of arbitrary polynomials the gap between upper and lower bounds
is  less dramatic than the exponential gap in the case of monic polynomials.

We note that the proof of  Theorem~\ref{thm:counting} is based on the close link
between the sets $\sDI_q = \sDI_q(1)$ and  $\sDI^*_q= \sDI_q^*(1)$, see Lemma~\ref{lem:monic} below.
On the other hand, for $r \ge 2$ there does not seem to
be any close relation between  $\sDI_q(r)$  and $\sDI_q^*(r)$. Accordingly, in this case
our result is weaker than for $r=1$.

We note that throughout the paper,  $o(1)$ denotes the quantity $\varepsilon(q)$,
which depends only on $q$ (and does not depend on $r$) with $\varepsilon(q)\to 0$
as $q \to \infty$.

\begin{theorem}\label{thm:dynamically-irreducible}
 Let $q$ be an odd prime power and $r\geq 2$. Then there exists an algorithm which computes $\DI_q(r)$ and constructs the set $\sDI_q(r)$ in time $ q^{r(3/2+o(1))+ 5/2}$ as $q\to \infty$ and uniformly over $r$.
 \end{theorem}

We give the pseudocode of the algorithm in Theorem~\ref{thm:dynamically-irreducible} in Algorithm~\ref{alg:2}.

As a by-product of the ideas behind our algorithm, we also obtain an analogue
of the upper bound~\eqref{eq:stable-bound}:

\begin{theorem}\label{thm:dynamically-irreducible bound}
 Let $q$ be an odd prime power and $r\geq 2$. Then
 $\DI_q(r) \le q^{r(3/2+o(1))+ 2}$
 as $q \to \infty$ and uniformly over $r$.
\end{theorem}

\section{Preliminaries}

We need to recall some important notions of the theory of dynamically irreducible quadratic polynomials, mainly
introduced by Jones and Boston~\cite{Jon,JB} (we recall that they are called `stable' in~\cite{JB,Jon}).

In particular, following~\cite{JB} we define the \emph{critical orbit} of $f$ as the set
$$
\Orb{f}=\{f^{(n)}(\gamma)~:~n=2,3,\ldots\}\subseteq \Fq,
$$
where $\gamma=-b/(2a)$ as the unique critical point of $f$.

We partition $\F_q$ into the sets of \emph{squares} $\cS_q$ and \emph{non-squares} $\cN_q$,
that is
$$
\cS_q = \{a^2~:~a\in \F_q\} \mand \cN_q = \F_q \setminus \cS_q.
$$

 We recall that for $a\in\Fq$ one can check whether $a \in \cN_q$ by
 evaluating its $(q-1)/2$-th power, as $a\in \cN_q$ if and only if $a^{(q-1)/2}=-1$.

By~\cite[Proposition~3]{JB}, a quadratic polynomial $f\in\Fq[X]$ is dynamically irreducible if the \emph{adjusted orbit}
$$
\aOrb{f}=\{-f(\gamma)\}\cup \Orb{f}
$$
satisfies
$$
\aOrb{f} \subseteq \cN_q.
$$

Clearly, the critical orbit $\Orb{f}$ of $f$ is a finite set. Furthermore,
by a result of Ostafe and Shparlinski~\cite{OstShp} the size of the critical orbit of a dynamically irreducible quadratic polynomial $f$
admits a nontrivial estimate. In particular, we now recall~\cite[Theorem~1]{OstShp}:

\begin{lemma}%[\color{red}Ostafe, Shparlinski \cite{OstShp}]
\label{lem:CritOrb-M}
There is an absolute  constant $c_1$  such that for
$$
M = \rf{c_1 q^{3/4}}
$$
for any $f \in \sDI_q$ we have
$$
\# \Orb{f} \le M.
$$
\end{lemma}

The following result reduces the problem of counting dynamically irreducible polynomials to  such dynamically irreducible 
polynomials where one of them is \emph{monic}.  It is a direct extension of~\cite[Lemma~2]{GoNi}, however for completeness, we sketch a proof.

\begin{lemma}\label{lem:monic}
For a polynomial $f\in\Fq[X]$ and $a\in\Fq^*$ define
$$
T_a(f)(X)=\frac{f(aX)}{a}\in\Fq[X].
$$
Then for any $a\in\Fq^*$,  $\{f_1,\ldots, f_r\}$ is dynamically irreducible if and only if $\{T_a(f_1),\ldots, T_a(f_r)\}$ is dynamically irreducible.
\end{lemma}

\begin{proof}
For an $a\in\Fq^*$, write $\xi_a(X)=aX$. First of all, we observe  that for $a\in\Fq^*$, $f$ is irreducible if and only if $\xi_a \circ f$ or $f \circ \xi_a $ is irreducible.

Assume, that $f_{i_1}\circ \ldots \circ f_{i_n}$ is irreducible for some $n$ and $i_1,\ldots, i_n\in\{1,\ldots,r\}$. Then  
$\xi_a^{-1}\circ f_{i_1}\circ \ldots f_{i_n}\circ \xi_a$ is irreducible, and therefore
\begin{align*}
\xi_a^{-1} \circ f_{i_1}&\circ \ldots f_{i_n}\circ \xi_a\\
& =\xi_a^{-1}\circ f_{i_1}\circ \xi_a \circ \xi_a^{-1}\circ f_{i_2} \circ \ldots f_{i_{n-1}}\circ \xi_{a}\circ \xi_{a}^{-1}\circ f_{i_n}\circ \xi_a\\
& =T_a\(f_{i_i}\) \circ  \ldots \circ T_a\(f_{i_n}\)
\end{align*}
is also irreducible. 
\end{proof}

In order to get the upper bound in the equation~\eqref{eq:stable-bound},
G\'omez-P\'erez  and  Nicol{\'a}s~\cite{GoNi} estimate the number of dynamically irreducible quadratic polynomials by the number of such polynomials that there is no square among the first $O(\log q)$ elements of their critical orbit. Their result can be summarized in the following way.

\begin{lemma}
\label{lem:CritOrb-K}
There is an absolute  constant $c_2$  such that for
$$
K=\rf{\frac{\log q}{2\log 2} +c_2}
$$
and
\begin{align*}
 \cF_K=\bigl\{f(X)=X^2+bX+c  \in  \F_q[X]~:&\\
 f^{(n)}(-b/2) \in \cN_q, & \ n=2,\ldots, K \bigr\}
\end{align*}
we have $\#\cF_K=O(q^{3/2}\log q)$.
\end{lemma}

In the following we extend some results of Ferraguti, Micheli and Schnyder~\cite{FMS1} and Heath-Brown and Micheli~\cite{H-BM} about dynamically irreducible sets for non-monic quadratic polynomials.

First we need the following result of Jones and Boston~\cite{JB}
(here we state the result in a corrected form, see~\cite{H-BM}).

\begin{lemma}\label{lemma:Jones&Boston} 
  Let $q$ be an odd prime and let $f(X)=aX^2+bX+c\in\Fq[X]$ and $\gamma=-b/(2a)$ be the unique critical point of $f$. Suppose that $g\in\Fq[X]$  has leading coefficient $e$, $g \circ f^{(n-1)}$ has degree $d$, and is irreducible over $\Fq$ for some $n\geq 1$. Then $g\circ f^{(n)}$ is irreducible over $\Fq[X]$ if and only if $(-a)^dg(f^{(n)}(\gamma))/e \in \cN_q$.
\end{lemma}

As a corollary, we get the characterization of dynamically irreducible sets of $r$ quadratic polynomials.

\begin{cor}\label{cor:test}
 Let $q$ be an odd prime. Let $f_i(X)=a_iX^2+b_iX+c_i\in\Fq[X]$ be irreducible quadratic polynomials for $1\leq i\leq r$. Write $\gamma_i=-b_i/(2a_i)$. Then $f_1,\ldots, f_r$ form a dynamically irreducible set if and only if for all integers $n\geq 1$ and $1\leq i_1,\ldots, i_n\leq r$ we have
 \begin{equation}\label{eq:general_req}
  a_{i_1}^{-1} (f_{i_1}\circ \ldots \circ f_{i_n})(\gamma_{i_n}) \in \cN_q.
 \end{equation}
\end{cor}
 
\begin{proof}
First assume that $f_1,\ldots, f_r$ form a dynamically irreducible set, that is, each iterate $f_{i_1}\circ \ldots \circ f_{i_n}$ with $n\geq 1$ and $1\leq i_1,\ldots, i_n\leq r$, is irreducible.
Applying Lemma~\ref{lemma:Jones&Boston} with $g=f_{i_1}\circ \ldots \circ f_{i_{n-1}}$ and $f=f_{i_n}$ we derive
\begin{equation}\label{eq:e}
e^{-1} (f_{i_1}\circ \ldots \circ f_{i_n})(\gamma_{i_n}) \in \cN_q,
\end{equation}
 where $e$ is the leading coefficient of $f_{i_1}\circ \ldots \circ f_{i_{n-1}}$. By induction, one can easily get, that $e=a_{i_1} a_{i_2}^2 \ldots a_{i_{n-1}}^{2^{n-2}}$.
Then~\eqref{eq:general_req} is  equivalent to~\eqref{eq:e}.

Conversely, if $f_{i_1}\circ \ldots \circ f_{i_n}$ is a reducible iterate with the smallest degree, then writing again $g=f_{i_1}\circ \ldots \circ f_{i_{n-1}}$ and $f=f_{i_n}$,
we see that~\eqref{eq:e} fails, which contradicts~\eqref{eq:general_req}.
\end{proof}

We remark that Corollary~\ref{cor:test} allows us to exhibit a large
family of  dynamically irreducible set of quadratic polynomials.

\begin{example}\label{ex:DI poly}
Let $q$ be a prime power such that $q\equiv 1\pmod 4$ 
and fix $b\in\F_q^*$. Let $f_a=a(X-b)^2+b$ for $a\in\F_q^*$. Then the set
$$\cF=\{f_{a}\mid ab \in \cN_q\}
$$
of cardinality $\#\cF=(q-1)/2$
is dynamically-irreducible.
 \end{example}

Indeed, let $r= (q-1)/2$ and take $a_1,\ldots,a_r$ such that $a_ib$ is a nonsquare  in $\F_q^*$ for all $1\le i\le r$.
We first notice that
$$
f_{a_{i_1}}\circ \cdots\circ f_{a_{i_{n}}}(X)=a_{i_1}a_{i_2}^2\cdots a_{i_n}^{2^{n-1}}(X-b)^{2^n}+b,\ 1\le i_1,\ldots,i_n\le r,
$$
and in particular that $f_{a_{i_1}}\circ \cdots\circ f_{a_{i_{n}}}(b)=b$. We apply now Corollary~\ref{cor:test}
to conclude that the set $\cF$ is dynamically-irreducible if and only if
$$
-a_{i_1}^{-1}b\in\cN_q\quad \mand \quad
a_{i_1}^{-1}b\in\cN_q\quad i_1=1,\ldots, r. 
$$
Since $q\equiv 1\pmod 4$,  we see that  $-1 \in \cS_q$ is a square, thus
the condition above is equivalent with $a_{i_1}b \in \cN_q$, which concludes our  argument.

Obviously, any subset of a set of dynamically irreducible  polynomials  $\{f_1,\ldots, f_r\}$ is also dynamically irreducible.
%Obviously, for any dynamically irreducible set $\{ f_1,\ldots, f_r\}$ of quadratic polynomials, its any subset  is also  dynamically irreducible. 
In particular each polynomial $f_i$ is  dynamically irreducible for $i =1,\ldots, r$.  Then by~\cite[Theorem~1]{OstShp}  we have the following result.

\begin{lemma}\label{lemma:iterates-arb} 
There is an absolute constant $c_3$ such that for
$$
 M=\lceil c_3 q^{3/4} \rceil
$$
the following holds.
If  $f_1,\ldots, f_r$ form a dynamically irreducible set of quadratic polynomials over a finite field $\Fq$ of odd characteristic and $\Gamma\subseteq \Fq$ such that
$$
 a_{i_1}^{-1} (f_{i_1}\circ \ldots \circ f_{i_n})(\gamma) \in \cN_q, \quad  n\geq 1,\ r \ge  i_1,\ldots, i_n\ge 1,
$$
for all $\gamma\in\Gamma$, then $\#\Gamma\leq M$.
\end{lemma}

Given $b,c\in\Fq$, it is convenient the set 
$$
\cF_{b,c}=\{f(X) = a(X-b)^2+ac+b~:~a\in\Fq^*  \}.
$$
Then the sets $\{\cF_{b,c} \}_{b,c}$   partition the set of all quadratic polynomials. We define the equivalent relation $\sim$ according to this partition, that is $f\sim g$ if $f,g\in\cF_{b,c}$ for some $b,c$.

If the polynomials $f_1,\ldots, f_r$ are not all equivalent, one can get a better bound than Lemma~\ref{lemma:iterates-arb}. For this we need a 
generalization  of a result of  Heath-Brown and Micheli~\cite[Lemma~2]{H-BM}, 
which also applies to non-monic polynomials.

\begin{lemma}\label{lemma:unique}
Let $f_1,f_2$ be quadratic polynomials over $\Fq$ such that
$f_1/f_2\not\in\Fq$ and $f_1\not\sim f_2$. 
If
 \begin{equation}\label{eq:composition-2}
  f_{i_1}\circ\ldots\circ f_{i_n}=\alpha f_{j_1}\circ\ldots\circ f_{j_m}
 \end{equation}
with $\alpha\in\Fq$, $i_1,\ldots, i_n,j_1,\ldots, j_m\in \{1, 2\}$, then $m=n$ and $i_h=j_h$ for  $h=1, \ldots, n$.
\end{lemma}
\begin{proof}
Write $f_i(X)=a_i(X-b_i)^2+c_i$, $i\in\{1,2\}$.

Assume, that~\eqref{eq:composition-2} holds. Then compering the degrees of both sides, we have $m=n$ and $\alpha\neq 0$. Moreover, as $f_1/f_2\not \in \Fq$, we can assume, that $n\geq 2$.
Write
$$
F_k=f_{i_k}\circ\ldots\circ f_{i_n} \mand G_k=f_{j_k}\circ\ldots\circ f_{j_n} \quad \text{for } 1\leq k \leq n,
$$
and let $d_k$ and $e_k$ ($1\leq k\leq n$) be the leading coefficients of $F_k$ and $G_k$ respectively. Clearly, 
$$
d_k=a_{i_k}a^{2}_{i_{k+1}}\ldots a^{2^{n-k}}_{i_n} \mand e_k=a_{j_k}a^{2}_{j_{k+1}}\ldots a^{2^{n-k}}_{j_n}.
$$
Put $d_{n+1}=e_{n+1}=1$, $b_{i_{0}}=b_{j_{0}}=0$, $ c_{i_{n+1}}=c_{j_{n+1}}=0$ and $F_{n+1}(X)=G_{n+1}(X)=X$. Then we claim, that
\begin{equation}\label{eq:ind-1}
  \frac{c_{i_{k}}-b_{i_{k-1}}}{d_{k}}=\frac{c_{j_{k}}-b_{j_{k-1}}}{e_{k}}\quad \text{for } 1\leq k \leq n+1,
\end{equation}
\begin{equation}\label{eq:ind-2}
 \frac{1}{d_{k+1}}F_{k+1}(X)-\frac{b_{i_{k}}}{d_{k+1}}=\frac{1}{e_{k+1}}G_{k+1}(X)-\frac{b_{j_{k}}}{e_{k+1}} \quad \text{for } 1\leq k \leq n
\end{equation}and
\begin{multline}\label{eq:ind-3}
 \left(\frac{1}{d_{k+1}}F_{k+1}(X)-\frac{b_{i_{k}}}{d_{k+1}} \right)^2+\frac{c_{i_{k}}-b_{i_{k-1}}}{d_{k}}\\
 =\left(\frac{1}{e_{k+1}}G_{k+1}(X)-\frac{b_{j_k}}{e_{k+1}} \right)^2+\frac{c_{j_k}-b_{j_{k-1}}}{e_{k}}\quad \text{for } 1\leq k \leq n.
\end{multline}

Indeed, if $F_1=\alpha G_1$, then
$$
d_1\left(\left(\frac{1}{d_{2}}F_{2}(X)-\frac{b_{i_1}}{d_{2}} \right)^2+\frac{c_{i_1}}{d_{1}}\right)=\alpha e_1\left(\left(\frac{1}{e_{2}}G_{2}(X)-\frac{b_{j_1}}{e_{2}} \right)^2+\frac{c_{j_1}}{e_{1}}\right).
$$
As both $d_2^{-1}F_2$ and $e_2^{-1}G_2$ are monic, we obtain
$$
\left(\frac{1}{d_{2}}F_{2}(X)-\frac{b_{i_1}}{d_{2}} \right)^2+\frac{c_{i_1}}{d_{1}}=\left(\frac{1}{e_{2}}G_{2}(X)-\frac{b_{j_1}}{e_{2}} \right)^2+\frac{c_{j_1}}{e_{1}}
$$
which proves~\eqref{eq:ind-3} for $k=1$. Now suppose, that~\eqref{eq:ind-3} holds for some $1\leq k\leq n$. Then

\begin{align*}
&\left( \frac{1}{d_{k+1}}F_{k+1}(X)- \frac{1}{e_{k+1}}G_{k+1}(X) -\frac{b_{i_{k}}}{d_{k+1}} +\frac{b_{j_{k}}}{e_{k+1}} \right)\\
&\quad \cdot\left( \frac{1}{d_{k+1}}F_{k+1}(X)+ \frac{1}{e_{k+1}}G_{k+1}(X) -\frac{b_{i_{k}}}{d_{k+1}} -\frac{b_{j_{k}}}{e_{k+1}} \right)\\
&=\frac{b_{i_{k-1}}-c_{i_{k}}}{d_{k}}-\frac{b_{j_{k-1}}-c_{j_{k}}}{e_{k}}.
\end{align*}
As $d_{k+1}^{-1}F_{k+1}$ and $e_{k+1}^{-1}G_{k+1}$ are monic, the second term of the left hand side has positive degree, thus
$$
\frac{b_{i_{k-1}}-c_{i_{k}}}{d_{k}}=\frac{b_{j_{k-1}}-c_{j_{k}}}{e_{k}}
$$
and
\begin{equation}\label{eq:ind-4}
\frac{1}{d_{k+1}}F_{k+1}(X)-\frac{b_{i_{k}}}{d_{k+1}}=\frac{1}{e_{k+1}}G_{k+1}(X)-\frac{b_{j_{k}}}{e_{k+1}}.
\end{equation}
which prove~\eqref{eq:ind-1} and~\eqref{eq:ind-2} for $k$. Moreover, if $k\leq n-1$, then from~\eqref{eq:ind-4} we get
\begin{align*}
\left(\frac{1}{d_{k+2}}F_{k+2}(X) -\frac{b_{i_{k+1}}}{d_{k+2}}  \right)^2&+\frac{c_{i_{k+1}}}{d_{k+1}}-\frac{b_{i_{k}}}{d_{k+1}}\\
&=\left(\frac{1}{e_{k+2}}G_{k+2}(X) -\frac{b_{j_{k+1}}}{e_{k+2}}  \right)^2+\frac{c_{j_{k+1}}}{e_{k+1}}-\frac{b_{j_{k}}}{e_{k+1}}
\end{align*}
which proves~\eqref{eq:ind-3} for $k+1$.

Finally,~\eqref{eq:ind-2} for $k=n$ proves~\eqref{eq:ind-1} for $k=n+1$.

To conclude the proof, let $1\leq k_0\leq n$ be the maximal index such that $f_{i_{k_0}}\neq f_{j_{k_0}}$. Then $d_k=e_k$ for $k\geq k_0+1$, thus by~\eqref{eq:ind-1} for $k=k_0+1$, we have that $b_1=b_2$.

% 
% 
% If $k_0=1$, then $F_{k}=G_{k}$ for $k\geq 2$, so we can write
% $$
% a_1\left((F_2(X)-b_1)^2+\frac{c_1}{a_1} \right)=\alpha a_2\left((G_2(X)-b_1)^2+\frac{c_2}{a_2} \right).
% $$
% Comparing the leading coefficients, we get $a_1=\alpha a_2$ and thus $c_1/a_1=c_2/a_2$ which means $f_1/f_2\in\Fq$, a contradiction. Then we can suppose, that $k_0\geq 2$.

As $d_{k_0}=a_{i_{k_0}}d^2_{k_0+1}$ and $e_{k_0}=a_{j_{k_0}}e^2_{k_0+1}$, by~\eqref{eq:ind-1} for $k=k_0$ we have
$$
\frac{c_{i_{k_0}}-b_{k_0-1}}{a_{i_{k_0}}}=\frac{c_{j_{k_0}}-b_{k_0-1}}{a_{j_{k_0}}}.
$$
If $k_0=1$, then $c_1/a_1=c_2/a_2$, which means $f_1/f_2\in\Fq$. If $k_0\geq 2$, then
% As $f_{i_{k_0-1}}\neq f_{j_{k_0-1}}$, we have
$$
\frac{c_{1}-b_1}{a_{1}}=\frac{c_{2}-b_1}{a_{2}},
$$
thus $f_1\sim f_2$.
\end{proof}

We now obtain a stronger version of Lemma~\ref{lemma:iterates-arb}
in the case when $f_1/f_2\not\in\Fq$ and $f_1\not\sim f_2$.

\begin{lemma}\label{lemma:iterates-nonconstant}
 There is an absolute constant $c_4$ such that for
 $$
  K=\left\lceil \sqrt{2\frac{\log \log q}{\log 2}}+c_4\right\rceil
 $$
the following holds.
If $f_1(X), f_2(X)\in\Fq[X]$ are two quadratic polynomials with $f_1/f_2\not\in\Fq$ and $f_1\not\sim f_2$
such that they form a dynamically-irreducible set and
$\Gamma\subseteq\Fq$ is a set with
\begin{equation}\label{eq:iterates-aa}
  a_{i_1}^{-1} (f_{i_1}\circ \ldots \circ f_{i_n})(\gamma) \in \cN_q, \quad n = 1,\ldots, K, \ i_1,\ldots, i_n\in \{1,2\} ,
\end{equation}
for all $\gamma\in\Gamma$, then $\#\Gamma \le q^{1/2}(\log q)^{1+o(1)}$.
\end{lemma}

\begin{proof}
 Put
$$
  \cF=\left\{a_{i_1}^{-1}f_{i_1}\circ \ldots \circ f_{i_n}: \ 1\leq n \leq K, \  i_1,\ldots, i_n\in \{1,2\}  \right\}.
$$

Let $\chi$ be the quadratic character of $\Fq$ and define $\chi(0)=1$, see \cite[Chapter~11]{IK} for a background of characters over finite fields.
Then, by~\eqref{eq:iterates-aa}
\begin{equation}\label{eq:Gamma}
 \#\Gamma\leq  \frac{1}{2^{\#\cF}}  \sum_{\gamma\in\Fq} \prod_{F\in \cF} (1-\chi(F(\gamma))).
\end{equation}

Expanding the products and rearranging the terms, we conclude that there are $2^{\#\cF}-1$ sums of form
\begin{equation}\label{eq:Gamma_2}
   \frac{(-1)^{\mu}}{2^{\#\cF}}  \sum_{\gamma\in\Fq} \chi \left( F_1(\gamma)\cdots F_\mu(\gamma)\right), \quad F_1,\ldots,F_\mu\in \cF,
\end{equation}
with some $\mu \le \#\cF$.
As $f_1,f_2$ form a dynamically irreducible set, the polynomials $F\in \cF$ are all irreducible. Moreover by Lemma~\ref{lemma:unique} they are coprime.
Thus the product polynomials $F_1\cdots F_\mu$ in~\eqref{eq:Gamma_2} are squarefree, which enables us to estimate~\eqref{eq:Gamma_2} by the Weil bound,  see~\cite[Theorem~11.23]{IK}. As $F_1\cdots F_\mu$ has degree at most $\mu\,2^K\leq \#\cF 2^K$ we we see from~\eqref{eq:Gamma} that 
$$
  \#\Gamma\leq \frac{q}{2^{\#\cF}}+\#\cF 2^K q^{1/2}.
$$
Using that
$$
\# \cF=2^{\binom{K+1}{2}}
$$
as all compositions in the definition of the set $\cF$ are distinct and choosing $K$ such that $2^{\# \cF}=O(q^{1/2})$ we obtain the result.
\end{proof}

Combining the algorithm of~\cite[Corollary~3]{H-BM} with Lemmas~\ref{lemma:iterates-arb}
and~\ref{lemma:iterates-nonconstant} one may get in the same way the following result.

\begin{lemma}\label{lem:algorithm} 
 There is an algorithm to test whether or not a set of $r$ quadratic polynomials over $\Fq$ is dynamically irreducible, which takes $O(rq^{3/4}\log q)$ operations. Moreover if the polynomials are not constant multiple of each other and not belong to the same equivalent class with respect to $\sim$, then the algorithm takes $O(rq^{1/2}(\log q)^3)$ operations.
\end{lemma}

 Finally, we also need the following result.

\begin{lemma}\label{lemma:small_graph}
 Let $f_1\in\Fq[X]$ be  a monic  quadratic polynomial with critical point $\gamma_1\in\Fq$. There are at most $O(q)$  quadratic polynomials $f_2$ with critical points $\gamma_2\in\Fq$ such that the set
 $$
 \cG(f_1,f_2)=\{f_{i_1}\circ \ldots \circ f_{i_n}(\gamma_{i_n})~:~n\geq 1,\  i_1,\ldots, i_n\in\{1,2\}  \}
 $$
has cardinality  $\#\cG(f_1,f_2)\leq 2$.
\end{lemma}

\begin{proof}
If $\#\cG(f_1,f_2)=1$, then $f_1$ has the form $f_1(X)=(X-\gamma_1)^2+\gamma_1$ as $\Orb{f_1}\subseteq G$. As $\Orb{f_2}\subseteq G=\{\gamma_1\}$, $f_2$ has the form $f_2(X)=a(X-\gamma_1)^2+\gamma_1$.

If $\#\cG(f_1,f_2)=2$, then its elements are solutions of $f_1(X)=X$ or $(f_1\circ f_1)(X)=X$ 
(as otherwise    $\#\cG(f_1,f_2)> 2$).
 Thus there are at most $O(1)$ choices for $\cG(f_1,f_2)$. Moreover, for a fixed $\cG(f_1,f_2)=\{g_1,g_2\}$, by the Lagrange interpolation, there are at most $O(q)$ quadratic polynomials $f_2$ such that $f_2(g_1),f_2(g_2)\in \{g_1,g_2\}$.
\end{proof}

\section{Proof of Theorem~\ref{thm:counting}}
We present an algorithm which computes the list of all quadratic dynamically irreducible polynomials, see Algorithm~\ref{alg:1}.

\begin{algorithm}[ht]
 \caption{Computing $\sDI_q$}\label{alg:1}
 
\begin{algorithmic}[1]
  \State let $M$ be as in Lemma~\ref{lem:CritOrb-M}
  \State $\sD\leftarrow []$,  $\sD^*\leftarrow []$
  \For{$b,c\in\Fq$} \label{line:alg1_monic_begin}
  \State $f\leftarrow X^2+bX+c$, $\gamma \leftarrow -b/2$
  \If{$-f(\gamma)\in\cN_q$ }
  \State $n\leftarrow 2$, $\delta \leftarrow f^{(2)}(\gamma)$
  \While{$n\leq M$ \and $\delta\in\cN_q$}
  \State  $n\leftarrow n+1$, $\delta \leftarrow f(\delta) $
  \EndWhile
  \If{$n=M$}
  \State append $f$ to $\sD^*$\label{line:alg1_monic_end}
  \EndIf
  \EndIf
  \For{$a\in\Fq^*$}\label{line:alg1_nonmonic_begin}
  \State append $T_a(f)$ to $\sD$\label{line:alg1_nonmonic_end}
  \EndFor
  \EndFor

  \Return $\sD$
\end{algorithmic}
\end{algorithm}
By Lemmas~\ref{lem:CritOrb-M}, the algorithm constructs all monic dynamically irreducible polynomials in Lines~\ref{line:alg1_monic_begin}-\ref{line:alg1_monic_end}, while in 
Lines~\ref{line:alg1_nonmonic_begin}-\ref{line:alg1_nonmonic_end}, it constructs all nonmonic dynamically irreducible polynomials by Lemma~\ref{lem:monic}.

To get the time complexity of Algorithm~\ref{alg:1}, let $K$ and $\cF_K$ be as in Lemma~\ref{lem:CritOrb-K}. Then the time complexity to obtain $\sD^*$ is
$$
q^2 \cdot K \cdot q^{o(1)}+ \# \cF_K \cdot M \cdot  q^{o(1)}\leq K q^{2+o(1)}+  M q^{3/2+o(1)} \leq q^{9/4+o(1)}.
$$
by Lemmas~\ref{lem:CritOrb-M} and~\ref{lem:CritOrb-K}.
To compute $\sD$ one needs further $\DI^*(q)\cdot q^{1+(1)}=q^{5/2+o(1)}$  many steps.

For computational reasons, the lists $\sD^*$ and $\sD$ do not have to be stored, thus to compute $\DI_q$ one needs to store just the length of these lists.

\section{Proof of Theorem~\ref{thm:dynamically-irreducible}}
We now present an algorithm which constructs all sets of $r$ pairwise distinct quadratic polynomials which are dynamically irreducible, see Algorithm~\ref{alg:2}. Throughout, we follow   the  convention, that  each polynomial $f_i$ is represented by its coefficients $f_i(X)=a_iX^2+b_iX+c_i$ and we also let $\gamma_i$ to be the critical point of $f_i$.

\begin{algorithm}[ht]
 \caption{Computing $\sDI_q(r)$}\label{alg:2}

 \algrenewcommand{\algorithmiccomment}[1]{\hfill \it \#{}#1}
\begin{algorithmic}[1]
\State $\sD(r)\leftarrow []$
\State compute $\sDI_q^*$  and $\sDI_q$ by Algorithm~\ref{alg:1}
\For{$f_1\in\sDI_q^*$} \label{line:f_1}
% \State let $\gamma_1$ be the critical point of $f_1$ and $a_1=1$ 
\State \begin{varwidth}[t]{\linewidth}
	$\cH\leftarrow \big\{f_2\in \sD_q: \#\{f_{i_1}\circ \ldots \circ f_{i_n}(\gamma_{i_n}): i_1,\ldots,i_n\in\{1,2\},\linebreak[4]
	\text{\hskip\algorithmicindent} 
	1\leq n\leq 3 \}\leq 2 \big\}$
       \end{varwidth}
\State $\cG\leftarrow \{a f_1:\ a\in\Fq^* \}\cup \cH\cup \cF_{\gamma_1,f_1(\gamma_1)-\gamma_1}$
\For{$(f_2,\ldots, f_r)\in \cG^{r-1}$, $f_i\neq f_j$ for $i\neq j$}\label{line:bad_polys_begin}
\If {$\{f_1, f_2,\ldots, f_r\}$ is dynamically irreducible}
\State append $\{f_1, f_2,\ldots, f_r\}$ to $\sD(r)$\label{line:bad_polys_end}
\EndIf
\EndFor
\State let $K$ as in Lemma~\ref{lemma:iterates-nonconstant}
\For{$f_2\in \sDI_q\setminus\cG$}\label{line_f_2_begin}

\If{$\{f_1,f_2\}$ is dynamically irreducible}
\State \begin{varwidth}[t]{\linewidth}
	$\Gamma \leftarrow \big\{\lambda\in\Fq: a_{i_1}^{-1} (f_{i_1}\circ \ldots \circ f_{i_n})(\lambda) \in \cN_q, \ 1\leq n\leq K, \linebreak[4] 
	 \text{\hskip\algorithmicindent} 
	i_1,\ldots, i_n\in \{1,2\} \big\}$
       \end{varwidth}
% \State let $\gamma_2$ be the critical point of $f_2$
\State \begin{varwidth}[t]{\linewidth}
        choose $\lambda_1,\lambda_2,\lambda_3\in  \{f_{i_1}\circ \ldots \circ f_{i_n}(\gamma_{i_n}): i_1,\ldots,i_n\in\{1,2\}, \linebreak[4]  
        \text{\hskip\algorithmicindent}  1\leq n\leq 3 \}$, $\lambda_i\neq\lambda_j$, $i\neq j$
       \end{varwidth}
%  \For{$3\leq i \leq r$}\label{line:polys_begin}
  \For{all choices of $r-2$ vectors $\alpha_3,\dots,\alpha_r\in \Gamma^3$}\label{line:polys_begin}
    \State \label{line:LES} solve the following for $a_i,b_i,c_i\in\Fq$, ($3\leq i\leq r$):
    \begin{align*}
    \begin{pmatrix}
\lambda_1^2 & \lambda_1 & 1\\ 
\lambda_2^2 & \lambda_2 & 1\\ 
\lambda_3^2 & \lambda_3 & 1\\ 
\end{pmatrix}
\begin{pmatrix}
a_i\\ 
b_i\\ 
c_i\\ 
\end{pmatrix}
=\alpha_i    
\quad \text{for } 3\leq i\leq r
%      a_i\lambda_1^2+b_i\lambda_1+c_i&=\alpha_1\\
%      a_i\lambda_2^2+b_i\lambda_2+c_i&=\alpha_2 \\ 
%      a_i\lambda_3^2+b_i\lambda_3+c_i&=\alpha_3 
    \end{align*}
    \State\label{line:polys_end} $f_i(X)\leftarrow a_iX^2+b_iX+c_i$
  \EndFor
% \EndFor
\If{$\{f_1,\ldots, f_r\}$ is dynamically irreducible}
\State append $\{f_1,\ldots, f_r \}$ to $\sD(r)$ \label{line_f_2_end}
\EndIf
\EndIf
\EndFor
\EndFor
\For{$a\in \Fq^*$ and $ \{f_1\ldots, f_r\}\in \sD(r)$ } \label{line:notmonic_begin}
\If{$ \{T_a(f_1)\ldots, T_a(f_r)\}\not\in \sD(r)$}
\State append $ \{T_a(f_1)\ldots, T_a(f_r)\}$ to $\sD(r)$ \label{line:notmonic_end}
\EndIf
\EndFor
\Return $\sD(r)$
\end{algorithmic}
\end{algorithm}

First we show the correctness of Algorithm~\ref{alg:2}.

Clearly $\sD(r)\subseteq \sDI(r)$. To show the equality, fix $\{f_1,\ldots, f_r\}\in \sDI(r)$. By Lines~\ref{line:notmonic_begin}-\ref{line:notmonic_end} we can assume, that $f_1$ is monic. If $\{f_1,\ldots, f_r\}$ is dynamically irreducible, then all of its subset also does. Specially, $f_1$ is dynamically irreducible monic polynomial, thus it is listed in Line~\ref{line:f_1}.

Put
$$
\cH=\big\{f_2\in \sD_q: \#\{f_{i_1}\circ \ldots \circ f_{i_n}(\gamma_{i_n}): i_1,\ldots,i_n\in\{1,2\}, 1\leq n\leq 3 \}\leq 2 \big\}
$$
It is easy to show, that $f_2\in\cH$ if and only if $G(f_1,f_2)\leq 2$. Define
$$
\cG= \{a f_1:\ a\in\Fq^* \}\cup \cH \cup\cF_{\gamma_1,f_1(\gamma_1)-\gamma_1}.
$$
If $f_2,\ldots,f_r\in\cG$, then the set $\{f_1,\ldots, f_r\}$ is considered in Lines~\ref{line:bad_polys_begin}-\ref{line:bad_polys_end}.
Thus, by rearranging the polynomials, we can assume, that $f_2\not\in\cG$.

Fix pairwise different $\lambda_1,\lambda_2,\lambda_3\in \{f_{i_1}\circ \ldots \circ f_{i_n}(\gamma_{i_n}): i_1,\ldots,i_n\in\{1,2\}, 1\leq n\leq 3 \}$
and define
\begin{equation*}
\Gamma=\left\{\gamma:\ a^{-1}_{i_1}f_{i_1}\circ \ldots \circ f_{i_n}(\gamma )\in\cN_q:\ i_1,\ldots, i_n\in \{1, 2\},\ 1\leq n\leq K\right \}.
\end{equation*}
As $\{f_1,\ldots, f_r\}$ is dynamically irreducible, $f_i(\lambda_1),f_i(\lambda_2), f_i(\lambda_3)\in \Gamma$ ($3\leq i\leq r$) by Corollary~\ref{cor:test}, thus the polynomials $f_3,\ldots,f_r$ appear as solutions of the system in Line~\ref{line:LES}. On the other hand, these systems are nonsingular as the coefficient matrix
$$
\begin{pmatrix}
\lambda_1^2 & \lambda_1 & 1\\ 
\lambda_2^2 & \lambda_2 & 1\\ 
\lambda_3^2 & \lambda_3 & 1\\ 
\end{pmatrix}
$$
is a Vandermonde with pairwise different $\lambda_1,\lambda_2,\lambda_3$.

Next, we estimate the time complexity of Algorithm~\ref{alg:2}. By Theorem~\ref{thm:counting} one can construct the sets $\sD_q^*$ and $\sD_q$ in time $q^{5/2+o(1)}$.

In Line~\ref{line:f_1} one can choose $f_1$ in   $\#\DI_q^*\leq q^{3/2+o(1)}$ ways by~\eqref{eq:stable-bound} and by Lemma~\ref{lem:monic}.

To construct the set $\cH$ one can check $\DI_q$ polynomials, which can be taken in time $O(\DI_q)=q^{5/2+o(1)}$. By Lemma~\ref{lemma:small_graph}, we have $\#\cH=O(q)$ thus
$$
\#\cG\leq q+\#\cH+\#\cF_{\gamma_1,f_1(\gamma_1)-\gamma_1}=O(q).
$$
  Hence  there are at most $q^{r-1+o(1)}$ polynomials $f_2,\ldots, f_r$ in Line~\ref{line:bad_polys_begin} to test, thus Lines~\ref{line:bad_polys_begin}-\ref{line:bad_polys_end} can be taken in time $rq^{r+1/4+o(1)}$ for a fixed $f_1$ by Lemma~\ref{lem:algorithm}.

Next, in Line~\ref{line_f_2_begin} we fix the polynomial $f_2\in \sDI_q\setminus \cG$. One can do this  in $\#(\sDI_q\setminus \cG)\leq q^{5/2+o(1)}$ ways. By Lemma~\ref{lem:algorithm}, one can check whether $\{f_1,f_2\}$ is dynamically irreducible in time $q^{1/2+o(1)}$ as $f_1/f_2\not\in \Fq$ and $f_1\not\sim f_2$.

One can  construct the set $\Gamma$ in time $q^{1+o(1)}2^K=q^{1+o(1)}$. Furthermore by Lemma~\ref{lemma:iterates-nonconstant}, we have $\#\Gamma \le q^{1/2+o(1)}$, thus one can construct the polynomials $f_i$ ($3\leq i \leq r$) in Lines~\ref{line:polys_begin}-\ref{line:polys_end} in 
\begin{equation}\label{eq:Gamma3}
(\#\Gamma)^{3}\leq q^{3/2+o(1)}
\end{equation}
ways. Finally, as $f_1/f_2\not\in\Fq$ and $f_1\not \sim f_2$, one can test if $\{f_1,\ldots, f_2\}$ is dynamically irreducible in time $rq^{1/2+o(1)}$ by Lemma~\ref{lem:algorithm}.

Summarizing, the time complexity of Lines~\ref{line:f_1}-\ref{line_f_2_end} is at most
\begin{align*}
& \DI_q^* \big( q^{5/2+o(1)}  +q+ \#\cG^{r-1}\cdot rq^{3/2+o(1)} \\
&\qquad +\DI_q\left( q^{1/2+o(1)}+q^{1+o(1)}+ (\#\Gamma)^{3(r-2)}rq^{1/2+o(1)} \right)\big)\\
&\leq rq^{3r/2+3/2+o(1)}
\end{align*}

Moreover, in Lines~\ref{line:f_1}-\ref{line_f_2_end} the algorithm construct at most
\begin{equation}\label{eq:f_1_monic}
\#\DI_q^* \left(\#\cG^{r-1} + \DI_q (\#\Gamma)^{3(r-2)}\right)\leq q^{3r/2+1+o(1)}
\end{equation}
dynamically irreducible polynomials $\{f_1,\ldots, f_r\}$ such that $f_1$ is monic. Thus the time complexity of Lines~\ref{line:notmonic_begin}-\ref{line:notmonic_end} is $q^{3r/2+2+o(1)}$.
This implies  that the complexity of Algorithm~\ref{alg:2} is
$$
rq^{3r/2+3/2+o(1)}+q^{3r/2+2+o(1)}\leq rq^{3r/2+3+o(1)}.
$$

\section{Proof of Theorem~\ref{thm:dynamically-irreducible bound}}
In the proof of Theorem~\ref{thm:dynamically-irreducible}, we have shown, that there are at most
$q^{3r/2+1+o(1)}$
dynamically irreducible polynomials $\{f_1,\ldots, f_r\}$ such that $f_1$ is monic, see~\eqref{eq:f_1_monic}.
Thus by Lemma~\ref{lem:monic}, there are at most $q^{3r/2+2+o(1)}$ set of dynamically irreducible polynomials of size $r$.

\section*{Acknowledgement}

The authors are very grateful to Alina Ostafe for several motivating discussions and providing the
construction of Example~\ref{ex:DI poly}.

Parts of this paper was written during visits of L.~M. and I.~S. to Max Planck Institute for Mathematics (Germany)
whose support and  hospitality are gratefully appreciated.

During the preparation of this  work D. G-P. is partially supported
by project MTM2014-55421-P from the Ministerio de Economia y Competitividad, 
L.~M. is partially supported by the Austrian Science Fund FWF Projects P30405 and F5511-N26
which is part of the Special Research Program ``Quasi-Monte Carlo Methods: Theory and Applications''
and  I.~S. by the Australian Research Council Grants DP170100786 and DP180100201

\end{document}